\documentclass[a4paper,10pt,onecolumn,twoside]{article}
\usepackage{mathrsfs}
\usepackage{fancyhdr}
\usepackage[numbers,sort&compress]{natbib}
\usepackage{amsmath,amsfonts,amssymb,graphicx,amsthm}
\allowdisplaybreaks
\usepackage{color}
\usepackage{hyperref} 
\usepackage{subfigure} 
\graphicspath{{figures/}}         

\usepackage{tikz} 
\usetikzlibrary{math,calc,intersections,through,backgrounds,animations,arrows,patterns,decorations.pathmorphing,decorations.pathreplacing,decorations.shapes,quotes,angles}  
\tikzset{global scale/.style={
    scale=#1,
    every node/.append style={scale=#1}
  }
}  

\usepackage{xcolor} 

\numberwithin{equation}{section}

 \newtheorem{lemma}{Lemma}[section]
 \newtheorem{proposition}[lemma]{Proposition}
 \newtheorem{theorem}{Theorem}[section]
 
 \theoremstyle{remark}
 \newtheorem{remark}{Remark}[section]

\numberwithin{equation}{section}

\begin{document}

\title{\bf Gradient Catastrophe for Solutions to the Conservation Laws with Source Term}

%

\author{{\sc Qingsong Zhao}\thanks{School of Mathematics and Science, Nanyang Institute of Technology, Nanyang 473004, China. Email: qqsszhao@nyist.edu.cn}}

\date{}

\maketitle
\begin{abstract}

This paper studies singularity formation for conservation laws with a source term. Motivated by John (1974) and Barlin (2023), we prove finite-time blow-up under initial data conditions weaker than those in Barlin. Moreover, we show that a sufficiently small compact support length of the initial data promotes blow-up. Hence, global existence can only be achieved when the initial data have a large compact support length.
\\[2mm]
\\[2mm]
\end{abstract}

\maketitle

\tableofcontents

%

\section{Introduction}

In this paper, we investigate the mechanism behind the formation of singularities of the conservation laws with source term of the form 
\begin{equation}\label{CLS}
\boldsymbol{u}_t+\boldsymbol{A}(\boldsymbol{u})\boldsymbol{u}_x={ \boldsymbol{g}(\boldsymbol{u})},
\end{equation}
where $\boldsymbol{A}:\Omega\to \mathbb{R}^{n\times n}$ and $\boldsymbol{g}:\Omega\to \mathbb{R}^{n}$ are smooth functions with $\Omega\in \mathbb{R}^n,$ satisfying:
\begin{enumerate}
\item [(A1)] $\boldsymbol{A}(\boldsymbol{0})$ has only real, simple eigenvalues;
\item [(A2)] all the eigenvalue of $\boldsymbol{A}$ is genuinely nonlinear at $\boldsymbol{u}=\boldsymbol{0}$;
\item [(A3)] $\boldsymbol{g}$ is a $C^2$ function, satisfying $\boldsymbol{g}(\boldsymbol{0})=\boldsymbol{0}$ and $\nabla_{\boldsymbol{u}}\boldsymbol{g}(\boldsymbol{0})=\boldsymbol{0}$.
\end{enumerate}
Source terms model a variety of physical effects such as external forces, relaxation toward equilibrium, chemical reactions, or source/sink terms. The presence of a source term introduces significant analytical and computational complexities, particularly in regimes where the source term is ``stiff" (i.e., changes much faster than the hyperbolic waves) or when solutions approach non-trivial equilibrium states.

In this paper we analyze whether blow-up behavior exists for the conservation laws \eqref{CLS} with source term when the initial data $\boldsymbol{u}(0,x)=\boldsymbol{u}_0(x)$ has compact support $\mathrm{supp}\ \boldsymbol{u}_0\subset [\alpha_0,\beta_0]$.

For the construction of blow-up solutions to hyperbolic conservation laws, we refer the readers to \cite{Barlin,Barlin-arxiv,Hormander,Hu-JDE-2022,Hu-MN-2019,John-JDE-1979,Lax,Li-Liu-2009,Liu,Zhao}. Among these results, the most classic is F. John's result in \cite{John-JDE-1979}. F. John  proved that if $\mathrm{(A1)}$ and $\mathrm{(A2)}$ hold, and the initial data $\boldsymbol{u}_0(x)\in C^2(\mathbb{R})$ has compact support, then, denoting
$$\mathrm{supp}\ \boldsymbol{u}_0\subset [\alpha_0,\beta_0]=I_0, \quad\quad s_0=\beta_0-\alpha_0<+\infty,\quad \quad
\theta=s_0^2\sup_{x}|\boldsymbol{u}_0^{\prime\prime}|,$$
there exists a positive constant $\theta_0=\theta_0(\boldsymbol{A},\boldsymbol{u}_0)$ such that when $0<\theta<\theta_0$, the solutions to the hyperbolic conservation law \eqref{CLS} without source term (i.e., $\boldsymbol{g}(\boldsymbol{u})=\boldsymbol{0}$) blow up in finite time. 

Taking the source term into account, J. B{\"a}rlin in \cite{Barlin} studied the blowup result of conservation laws \eqref{CLS}. Let the initial data $\boldsymbol{u}_0$ take the form
\begin{equation}\label{initial-data}
\boldsymbol{u}_0(x)=\boldsymbol{U}(\varepsilon\alpha(\varepsilon^{-1}\kappa^{-1}x)),
\end{equation}
where $\boldsymbol{U}$ satisfies $\boldsymbol{U}(0)=\boldsymbol{0}$ and the differential equation
\begin{equation}\label{eq-U}
\boldsymbol{U}^\prime(\xi)=\boldsymbol{r}_p(\boldsymbol{U}(\xi)).
\end{equation}
Here the function $\alpha\in C_c^\infty(\mathbb{R})$ satisfies ${\rm supp}\ \alpha\subset \left(-\frac12,\frac12\right)$ and $\max \alpha^\prime>0.$ The constants $\varepsilon$ and  $\kappa$ are sufficiently small. Then the solution to the hyperbolic conservation law  with source term \eqref{CLS} and initial data \eqref{initial-data} blows up in finite time, and its maximum existence time $T$ satisfies 
\begin{equation*}
T<T_\kappa:=\frac34\kappa\overline{T}.
\end{equation*}
Here $\overline{T}=\frac{4}{\gamma_{ppp}(\boldsymbol{0})\max \alpha^\prime}.$

In this paper, we first extend the initial data \eqref{initial-data} and remove the restriction which needs that $\boldsymbol{U}$ satisfies the differential equations \eqref{eq-U}.

\begin{theorem}\label{thmmain}
Suppose that $\boldsymbol{u}_0(x)\in C^2(\mathbb{R})$ has compact support and there exists a positive constant $\delta$ such that $\|\boldsymbol{u}^\prime_0(x)\|_{L^\infty(\mathbb{R})}\geq \delta.$  Denote
$$\mathrm{supp}\ \boldsymbol{u}_0\subset [\alpha_0,\beta_0]=I_0, \quad\quad s_0=\beta_0-\alpha_0<+\infty,$$
$$\theta_0=\left(1+s_0^{-\frac12}\right)\left(1+\left(\sup_{x}|\boldsymbol{u}_0^\prime(x)|\right)^{-1}\right)s_0^2\sup_{x}|\boldsymbol{u}_0^{\prime\prime}(x)|.$$
Then there exists a positive number $\theta_\ast=\theta_\ast(\boldsymbol{A},\boldsymbol{u}_0  ),$ such that if  $0<\theta_0<\theta_\ast$, the solution of the conservation laws \eqref{CLS} with initial data $\boldsymbol{u}(0,x)=\boldsymbol{u}_0(x)$ blows up in finite time.
\end{theorem}

\begin{remark}
The smallness of $\theta_0$ implies that $s_0=\beta_0-\alpha_0$ is small. More precisely, notice that
\begin{align*}
s_0^2\|\boldsymbol{u}^{\prime\prime}_0\|_{L^\infty(\mathbb{R})}=&\theta_0\left(1+s_0^{-\frac12}\right)^{-1}\left(1+\|\boldsymbol{u}_0^\prime\|^{-1}_{L^\infty(\mathbb{R})}\right)^{-1}\\
\leq& \theta_0 s_0^{\frac12}\|\boldsymbol{u}_0^\prime\|_{L^\infty(\mathbb{R})}
\leq \theta_0 s_0^{\frac32}\|\boldsymbol{u}_0^{\prime\prime}\|_{L^\infty(\mathbb{R})}.
\end{align*}
Hence, we have 
\begin{equation}\label{small-theta}
s_0\leq \theta_0^2.
\end{equation}
We shall explain why the smallness of $s_0$ is indispensable in Theorem \ref{thm:2}.
\end{remark}
 
\begin{remark}
For any $\alpha\in C_c^\infty(\mathbb{R}),$  $\mathrm{supp}\ \alpha(x)\subset [-\frac12,\frac12]$ and $\boldsymbol{U}\in C^2(\mathbb{R}),$ if we take $\boldsymbol{u}_0(x)=\boldsymbol{U}(\varepsilon\alpha(\varepsilon^{-\ell}x)),$ where $\ell>1,$ 
it holds that
\begin{equation*}
s_0\lesssim \varepsilon^\ell,\quad\quad\quad|\boldsymbol{u}_0^\prime|\lesssim \varepsilon^{1-\ell},\quad\quad\quad |\boldsymbol{u}_0^{\prime\prime}|\lesssim \varepsilon^{1-2\ell}.
\end{equation*}
Then $$\theta_0\lesssim \varepsilon^{\frac{\ell}2},$$
which implies $\theta_0$ goes to zero as long as $\varepsilon$ is small enough. Hence, $\theta_*$ can be chosen as
\begin{equation*}
\theta_*=C_*\varepsilon^{\frac{\ell }{2}},
\end{equation*}
where $C_*$ is a constant related to $\boldsymbol{A}(\boldsymbol{u}),\boldsymbol{U}(x)$ and $\alpha(x).$
\end{remark}
\begin{remark}
For the initial data with small $C^2-$norm, it is known that the solutions to the conservation laws exist globally. Therefore, when considering blow-up solutions, a natural assumption --- as in \cite{John-JDE-1979} --- is that the initial function $\boldsymbol{u}_0(x)$ itself is small, while its derivative $\boldsymbol{u}_0^\prime(x)$ is large. Hence, we assume in Theorem \ref{thmmain} that $\|\boldsymbol{u}_0^\prime(x)\|_{L^\infty(\mathbb{R})}$ has positive lower bound. 
\end{remark}
 
A classic example is the scalar equation 
$$u_t+uu_x=u^{1+\ell},$$
where $\ell>0.$ The source term on the right-hand side is superlinear. One can construct blow-up solutions of Barenblatt-type or self-similar type, proving  that for specific initial data, the solution blows up in finite time\cite{Chen-Yang-AA-2024}.

It is well known that John’s classical blow-up result for solutions to one dimensional conservation laws in \cite{John-JDE-1979} does not require any smallness condition on the length $s_0$ of the compact support of the initial data. However, both B\"{a}rlin’s work \cite{Barlin} and the present paper impose a smallness condition on $s_0$. Next, we will take a single conservation law as an example to explain in detail the reason for this. Actually, we give a necessary and sufficient condition for the global existence or finite time blowup for the solution  of a single conservation law with source term along a characteristic curve.

Let $k>0.$ Consider the following conservation law with a source term
\begin{equation}\label{CL1}
u_t+uu_x=ku^2,
\end{equation}
whose initial data 
\begin{equation}\label{initial}
u(0,x)=u_0(x)
\end{equation}
is a non-identically-zero function with compact support in the interval $[\alpha_0,\beta_0].$ Let $X(t,z)$ be the characteristic curve starting from $z$. 
We show that the smallness of $s_0$ is indispensable.

\begin{theorem}\label{thm:2}
Let $u_0(x)$ be a non-identically-zero function with compact support in the interval $[\alpha_0,\beta_0].$
Let $s_0=\beta_0-\alpha_0$ and $z_0$ be a global minimum point of $u_0^\prime(x)$. Denote by $u(t,x)$ the solution of equation \eqref{CL1} with initial datum $u_0(x).$ 
\begin{enumerate}
\item [(1)] If $s_0\leq k^{-1},$ then $u(t,x)$ or $u_x(t,x)$ must blow up on the characteristic curve $X(t,z_0)$.
\item [(2)] If $s_0> k^{-1},$ then there exists a certain $u_0(x),$ such that $u(t,x)$ and $u_x(t,x)$ exist globally on the characteristic curve $X(t,z_0)$.
\end{enumerate}
\end{theorem}
\begin{remark}

Theorem \ref{thm:2} states that for a global solution to the conservation law \eqref{CL1} to exist, the initial data must have a sufficiently large compact support length; otherwise, if the compact support length is small enough, the solution necessarily blows up in finite time.
\end{remark}

Our chapter arrangement is as follows. In Section 2, we explore what properties of the initial data can be obtained from the smallness of $\theta_0$ in Theorem \ref{thmmain} and the classic results on hyperbolic conservation laws. In Section 3, we give the proof of Theorem \ref{thmmain}. In Section 4, we give the proof of Theorem \ref{thm:2}.

\noindent\textit{Notations.} The inequality  $f\lesssim g$ means $f(x)\leq Cg(x),$ where $C$ is a constant depending only on $\boldsymbol{A}$ and the initial data $\boldsymbol{u}_0(x).$

\section{Preliminary}
In this section, we review the classic results on hyperbolic conservation laws.



At the beginning of this section, we explore what properties of the initial data can be obtained from the smallness of $\theta_0$ in Theorem \ref{thmmain}. For the initial data $\boldsymbol{u}_0,$ we define
\begin{equation*}
W_0=\sup_{i,x}|w_i(0,x)|,\quad\quad\quad W_0^+=\sup_{i,x}w_i(0,x).
\end{equation*}
Let 
\begin{equation*}
\theta_2=s_0^2\sup_x|\boldsymbol{u}_0^{\prime\prime}|,\quad\quad\quad \theta_1=\left(1+\left(\sup_{x}|\boldsymbol{u}_0^\prime(x)|\right)^{-1}\right)\theta_2,
\end{equation*}
then it is obvious that
\begin{equation}\label{2}
|\boldsymbol{u}_0^{\prime\prime}|\leq \theta_2 s_0^{-2},\quad\quad\quad 
|\boldsymbol{u}_0^{\prime}|\leq \theta_2 s_0^{-1},\quad\quad\quad 
|\boldsymbol{u}_0|\leq \theta_2.
\end{equation}

We now give the relationship between $\theta_i \ (i=0,1,2)$ and the constants $W_0$ and $W_0^+.$
\begin{lemma}\label{lem:prop-theta}
The constants $\theta_0, \theta_1$, $\theta_2$ and $W_0,$ $W_0^+$ satisfy the following properties:
\begin{equation}\label{dif-theta}
\theta_0\geq\theta_1\geq\theta_2,
\end{equation}
\begin{equation}\label{bound-W}
s_0W_0^2\lesssim \theta_2W_0^+,
\end{equation}
and
\begin{equation}\label{bound-theta2}
\theta_2\lesssim \theta_0^2 W_0^+,
\end{equation}
\end{lemma}
\begin{proof}
Since \eqref{dif-theta} is obvious and \eqref{bound-W} has been proved in \cite{John-JDE-1979}, we only give the proof of \eqref{bound-theta2}. From \eqref{bound-W} , the definition of $\theta_0$, and the fact that
\begin{equation*}
(1+x^{-1})^{-1}\leq x,\quad\quad\quad \forall x>0,
\end{equation*}
the inequality \eqref{bound-theta2} can be obtained by noticing that
\begin{equation*}
\frac{\theta_2}{\theta_0}=\left(1+s_0^{-\frac12}\right)^{-1}\left(1+W_0^{-1}\right)^{-1}\leq\sqrt{s_0W_0^2}\lesssim\sqrt{\theta_2}\sqrt{W_0^+}.
\end{equation*}
\end{proof}

In the proof of Theorem \ref{thmmain}, the term $s_0W_0+s_0W_0^2$ appears frequently. Hence, we give some estimates corresponding to the term $s_0W_0+s_0W_0^2$ in the next lemma.
\begin{lemma}\label{lem:sW+sW2}
The term $s_0W_0+s_0W_0^2$ satisfies:
\begin{align}
s_0W_0+s_0W_0^2&\lesssim \theta_0, \label{sW-1} \\
s_0W_0+s_0W_0^2&\lesssim \theta_1 W_0^+,  \label{sW-2}\\
s_0W_0+s_0W_0^2&\lesssim \theta_0^2(W_0^+)^2.\label{sW-3}
\end{align}
\end{lemma}
\begin{proof}
By the definition of $\theta_1,$ we have
\begin{equation*}
\theta_1=\left(1+s_0^{-\frac12}\right)^{-1}\theta_0\leq s_0^{\frac12}\theta_0.
\end{equation*}
From \eqref{2} and \eqref{dif-theta}, one has
\begin{equation*}
s_0W_0^2\leq s_0^{-1}\theta_2^2\leq s_0^{-1}\theta_1^2\leq \theta_0,
\end{equation*}
which combining with \eqref{2} implies \eqref{sW-1}.

By the relationship between $\theta_1$ and $\theta_2,$ we have from \eqref{bound-W} that
\begin{equation*}
s_0W_0^2\lesssim \theta_2W_0^+\lesssim (1+W_0^{-1})^{-1}\theta_1W_0^+\lesssim W_0\theta_1W_0^+.
\end{equation*}
Eliminating $W_0$ yields
\begin{equation}\label{3}
s_0W_0\lesssim \theta_1 W_0^+.
\end{equation}
Combining the above estimate with \eqref{bound-W} yields \eqref{sW-2}.

Notice that from \eqref{bound-W} and \eqref{bound-theta2}, one has
\begin{equation*}
s_0W_0^2\lesssim \theta_0^2(W_0^+)^2.
\end{equation*}
In order to prove \eqref{sW-3}, it suffices to prove that
\begin{equation}\label{4}
s_0W_0\lesssim \theta_0^2(W_0^+)^2.
\end{equation}
By the definition of $\theta_1$ and $\theta_2,$ and $W_0>\delta,$
we have from \eqref{3} that
\begin{equation*}
\frac{s_0W_0}{(W_0^+)^2}\lesssim \frac{\theta_1}{W_0^+}=\frac{(1+W_0^{-1})\theta_2}{W_0^+}\lesssim \frac{\theta_2}{W_0^+}.
\end{equation*}
This and \eqref{bound-theta2} imply \eqref{4}.
\end{proof}

Then, we introduce the assumption that $|\boldsymbol{u}|\leq \delta_1,$ which will now be treated as {\it a priori assumption}: the solution $\boldsymbol{u}\in C^2([0,T]\times \mathbb{R})$ of the conservation laws with source term \eqref{CLS} satisfies
\begin{equation}\label{pri-ass}
|\boldsymbol{u}(t,x)|\leq \delta_1,\quad\quad t\in[0,T], x\in\mathbb{R}.
\end{equation}

From the smoothness of the matrix $\boldsymbol{A}(\boldsymbol{u})$ and the condition (A1),  it follows that for a small positive constant $\delta_1$, the matrix $\boldsymbol{A}=\boldsymbol{A}(\boldsymbol{u})$ is strictly hyperbolic and has distinct real eigenvalues  $\lambda_i=\lambda_i(\boldsymbol{u}),$ satisfying 
\begin{equation*}
\lambda_1(\boldsymbol{u})<\lambda_2(\boldsymbol{u})<\cdots<\lambda_n(\boldsymbol{u})
\end{equation*}
for all $\boldsymbol{u}$ such that $|\boldsymbol{u}|=\sqrt{\boldsymbol{u}\cdot \boldsymbol{u}}\leq \delta_1.$

Suppse that for the eigenvalue $\lambda_i(\boldsymbol{u})$ of the matrix $\boldsymbol{A}(\boldsymbol{u}),$ the row vector $\boldsymbol{\ell}_i=\boldsymbol{\ell}_i(\boldsymbol{u})$ is a left eigenvector and the column vector $\boldsymbol{r}_i=\boldsymbol{r}_i(\boldsymbol{u})$ is a right eigenvector, satisfying 
\begin{equation}\label{eigen-v}
\|\boldsymbol{\ell}_i\|=1,\quad\quad \boldsymbol{\ell}_i\cdot\boldsymbol{r}_j=\delta_{ij}.
\end{equation}
Then any row vector $\boldsymbol{\ell}$ and any column vector $\boldsymbol{r}$ can be decomposed as
\begin{equation*}
\boldsymbol{\ell}=\sum_{k=1}^n(\boldsymbol{\ell}\cdot\boldsymbol{r}_k)\boldsymbol{\ell}_k,\quad\quad \boldsymbol{r}=\sum_{k=1}^n(\boldsymbol{\ell}_k\cdot\boldsymbol{r})\boldsymbol{r}_k.
\end{equation*}

Denote by 
$$\boldsymbol{w}=\boldsymbol{u}_x,\quad\quad\quad w_k=\boldsymbol{\ell }_k\cdot \boldsymbol{w}.$$
Then, the first-order derivative $\boldsymbol{w}=\boldsymbol{u}_x$ can be decomposed as
\begin{equation*}
\boldsymbol{w}=\sum_{k=1}^n(\boldsymbol{\ell}_k\cdot\boldsymbol{w})\boldsymbol{r}_k=\sum_{k=1}^nw_k\boldsymbol{r}_k.
\end{equation*}

Now we define the characteristic curves of the conservation laws \eqref{CLS}. Let $T>0.$ The $i$-characteristic curve 
$$\mathcal{C}_i(z)=\{(t,x):0\leq t\leq T, x=X_i(t,z)\}$$
 starting from $(0,z)$ is the curve determined by the ordinary differential equation
\begin{equation*}
\frac{\partial }{\partial t}X_i(t,z)=\lambda_i(\boldsymbol{u}(t,X_i(t,z))),\quad\quad\quad X_i(0,z)=z.
\end{equation*}

Define
$$
c_{ijk}=c_{ijk}(\boldsymbol{u})=\boldsymbol{l}_i\boldsymbol{C}_k\boldsymbol{r}_j,$$
where
\begin{equation*}
\boldsymbol{C}_k=\boldsymbol{C}_k(\boldsymbol{u}):=\frac{\mathrm{d}}{\mathrm{d}s}\boldsymbol{A}(\boldsymbol{u}+s\boldsymbol{r}_k)\bigg|_{s=0}
\end{equation*}
represents the directional derivative of the matrix $\boldsymbol{A}$ with respect to the direction $\boldsymbol{r}_k$.  Then, it holds that
\begin{equation}\label{nabla-lambda}
\mathrm{d}\lambda_i=\sum_kc_{iik}\left(\boldsymbol{l}_k\cdot \mathrm{d}\boldsymbol{u} \right).
\end{equation}

For all $i\in \{1,2,\cdots,n\},$ the following lemma shows that the function $w_i(t,x)$ along the characteristic curve $\mathcal{C}_i(z)$ satisfies a certain ordinary differential equation along the characteristic curves.
\begin{lemma}[cf. John \cite{John-JDE-1979}, B\"{a}rlin \cite{Barlin}]\label{lem:s-eq-w}
The function $w_i$ satisfies the ordinary differential equation 
\begin{equation}\label{eq-w-s}
\frac{\mathrm{d}w_i}{\mathrm{d}t}=\sum_{j,k}\gamma_{ijk}w_jw_k+{\sum_{k}g_{ik}w_k}
\end{equation}
along the characteristic curve $\mathcal{C}_i(z)$, where
\begin{equation}\label{gamma-ijk}
\sum_{j,k}\gamma_{ijk}w_jw_k=\sum_{\substack{j,k\\j\neq i}}\frac{\lambda_i-\lambda_k}{\lambda_j-\lambda_i}c_{ijk}w_k\left[(\boldsymbol{l}_j\cdot\boldsymbol{l}_i^T)w_i-w_j\right]-\sum_{j,k}c_{ijk}w_jw_k,
\end{equation}
\begin{equation}\label{gik}
\sum_kg_{ik}w_k=\sum_k\left(\boldsymbol{l}_i\cdot\nabla_{\boldsymbol{u}}\boldsymbol{g}\cdot\boldsymbol{r}_k\right)w_k+\sum_{\substack{j,k\\j\neq i}}\frac{1}{\lambda_j-\lambda_i}c_{ijk}\left(\boldsymbol{l}_k\cdot \boldsymbol{g}\right)\left[\left(\boldsymbol{l}_j\cdot\boldsymbol{l}_i^T\right)w_i-w_j\right].
\end{equation}
Here $\gamma_{ijk}=\gamma_{ijk}(\boldsymbol{u})$ satisfy
\begin{subequations}\label{p-gamma}
\begin{align}
\gamma_{ijk}=&\gamma_{ikj}, \label{p-1} \\
\gamma_{iii}=&-c_{iii},  \label{p-2}  \\
\gamma_{ijj}=&0,\quad (j\neq i), \label{p-3}\\
2\gamma_{iij}=&\sum_{m\neq i}\frac{\lambda_i-\lambda_j}{\lambda_m-\lambda_i}c_{imj}(\boldsymbol{l}_m\cdot\boldsymbol{l}_i^T)-c_{iij}-c_{iji},\quad (j\neq i), \label{p-4}\\
2\gamma_{ijk}=&-\frac{\lambda_j-\lambda_k}{\lambda_j-\lambda_i}c_{ijk}-\frac{\lambda_k-\lambda_j}{\lambda_k-\lambda_i}c_{ikj},\quad (j\neq i, k\neq i), \label{p-5}
\end{align}
\end{subequations}
and $g_{ik}=g_{ik}(\boldsymbol{u})$ satisfy
\begin{subequations}\label{p-g}
\begin{align}
g_{ii}=&\boldsymbol{l}_i\cdot\nabla_{\boldsymbol{u}}g\cdot\boldsymbol{r}_i+\sum_{\substack{j,k\\k\neq j}}\frac{1}{\lambda_k-\lambda_i} c_{ikj}\left(\boldsymbol{l}_j\cdot \boldsymbol{g} \right)(\boldsymbol{l}_k\cdot\boldsymbol{l}_i^T),\label{pg-1}\\
g_{ik}=&\boldsymbol{l}_i\cdot\nabla_{\boldsymbol{u}}g\cdot\boldsymbol{r}_k-\sum_{j}\frac{1}{\lambda_k-\lambda_i}  c_{ikj}\left(\boldsymbol{l}_j\cdot \boldsymbol{g} \right),\quad\quad\quad (k\neq i).\label{pg-2}
\end{align}
\end{subequations}
\end{lemma}

By \eqref{p-2}, the assumption (A2) and the definition of genuine nonlinearity, we have
\begin{equation*}
\gamma_{iii}(\boldsymbol{u}) = -c_{iii}(\boldsymbol{u}) = -\frac{\mathrm{d}}{\mathrm{d}s}\lambda_i(\boldsymbol{u}+s\boldsymbol{r}_i)\bigg|_{s=0} = -\nabla_{\boldsymbol{u}}\lambda_i(\boldsymbol{u})\cdot\boldsymbol{r}_i(\boldsymbol{u}) \neq 0.
\end{equation*}
Note that when $\boldsymbol{l}_i$ is replaced by $-\boldsymbol{l}_i$, $\boldsymbol{r}_i$ also changes to $-\boldsymbol{r}_i$; then the sign of the above expression changes accordingly. Hence, by choosing the sign of $\boldsymbol{l}_i$ appropriately, we can always make 
\begin{equation}\label{g>0}
\gamma_{iii}(\boldsymbol{u}) > 0 
\end{equation}
for all $|\boldsymbol{u}|\leq \delta_1$.

Let the family of $i$-characteristic curves starting from $\alpha_0$ and $\beta_0$ be denoted by
\begin{equation*}
\alpha_i(t)=X_i(t,\alpha_0),\quad\quad\quad \beta_i(t)=X_i(t,\beta_0).
\end{equation*}
Define as $R_i$ the characteristic strip formed by all the $i-$th family of characteristic curves emanating from the interval $I_0=[\alpha_0,\beta_0]$ (see Figure \ref{pic:1}), i.e., 
\begin{equation*}
R_i:=\cup_{z\in I_0}\mathcal{C}_i(z)=\{(t,x):\alpha_i(t)\leq x\leq \beta_i(t),\ 0\leq t\leq T\}.
\end{equation*}

\begin{figure}[htbp]
\centering
\begin{tikzpicture}[samples=50, global scale=0.65]
\draw [thick] (-7,0) -- (-1,0) node [below] {$\alpha_0$} -- (1,0)  node  [below] {$\beta_0$} -- (7,0) node [right] {$t=0$} ;
\draw [thick] (-7,10) -- (7,10) node [right] {$t=T$} ;
\draw [dashed] (-7,5) -- (7,5) node [right] {$t=t_0$};
\draw 
(-6.5,4.4) node (a1) {$\mathcal{C}_1(\alpha_0)$} 
(-4.2,7.3) node (b1) {$\mathcal{C}_1(\beta_0)$}
(-5,8) node (a2) {$\mathcal{C}_2(\alpha_0)$} 
(-1,9.2) node (b2) {$\mathcal{C}_2(\beta_0)$}
(1.5,9.4) node (a3) {$\mathcal{C}_3(\alpha_0)$} 
(6.2,8.1) node (b3) {$\mathcal{C}_3(\beta_0)$}
(4.5,7) node (a4) {$\mathcal{C}_4(\alpha_0)$} 
(6.1,4.5) node (b4) {$\mathcal{C}_4(\beta_0)$};
\draw [bend right,distance=20mm, thick] (-1,0) to node [pos=0.8] (11) {} (-7,6);  
\draw [bend right,distance=25mm, thick] (1,0) to node [pos=0.8] (12)  {} (-7,7);  
\draw [bend right,distance=20mm, thick] (-1,0) to node [pos=0.8] (21) {} (-5,10);  
\draw [bend right,distance=25mm, thick] (1,0) to node [pos=0.8] (22) {} (-3.5,10);  
\draw [bend left,distance=20mm, thick] (-1,0) to node [pos=0.8] (31) {} (4,10);  
\draw [bend left,distance=25mm, thick] (1,0) to node [pos=0.8] (32) {} (6.5,10);  
\draw [bend left,distance=20mm, thick] (-1,0) to node [pos=0.8] (41) {} (7,7);  
\draw [bend left,distance=25mm, thick] (1,0) to node [pos=0.8] (42)  {} (7,6);  
\draw[->] (a1)--(11);
\draw[->] (b1)--(12);
\draw[->] (a2)--(21);
\draw[->] (b2)--(22);
\draw[->] (a3)--(31);
\draw[->] (b3)--(32);
\draw[->] (a4)--(41);
\draw[->] (b4)--(42);
\node (R1) at (-3,4) {$R_1$};
\node (R2) at (-2,6.2) {$R_2$};
\node (R3) at (2,6) {$R_i$};
\node (R4) at (3.1,4) {$R_n$};
\node (a) at (0.2,7.5) {$\cdots$};
\node (b) at (5.5,7.5) {$\ddots$};
\end{tikzpicture}
\caption{The characteristic strip $R_i$}
\label{pic:1}
\end{figure}
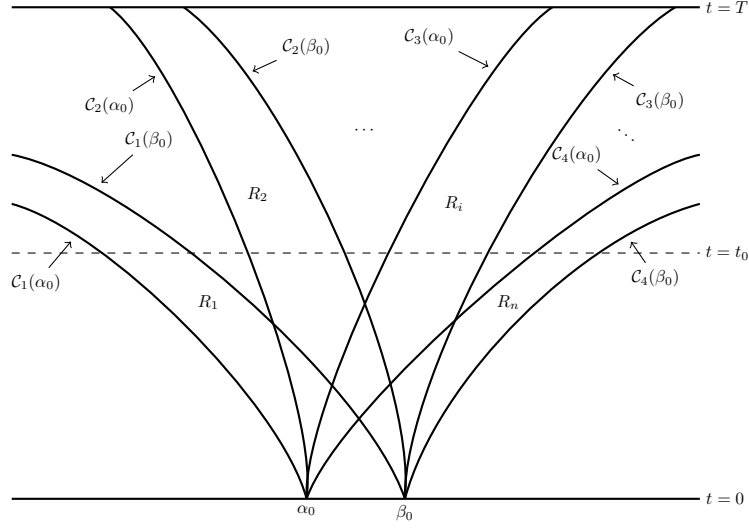

Since $\boldsymbol{g}(\boldsymbol{0})=\boldsymbol{0}$ and $\boldsymbol{g}$ is identically zero outside a region of compact support, we thus obtain the following conclusion regarding the finite speed of propagation by applying the method of characteristics.
\begin{lemma}[cf. B\"{a}rlin \cite{Barlin}]\label{lem:finite-velocity}
Suppose that $\boldsymbol{u}\in C^1([0,T]\times\mathbb{R})$ is a solution of the conservation laws \eqref{CLS}, satisfying
\begin{equation*}
{\rm supp}\ \boldsymbol{u}(0,x)\subset [\alpha_0,\beta_0],
\end{equation*}
then for any $t\in [0,T],$ 
\begin{equation*}
{\rm supp}\ \boldsymbol{u}(t,\cdot)\subset [\alpha_0+\lambda_1(\boldsymbol{0})t,\ \beta_0+\lambda_n(\boldsymbol{0})t].
\end{equation*}
\end{lemma}

Since the source term $\boldsymbol{g}(\boldsymbol{u})$ is a lower order term in \eqref{CLS}, it does not cause any problem in the proof of the local existence result, cf. \cite{Barlin,Hormander,Lax-1973,Rauch-2012}. Hence, we have
\begin{lemma}\label{lem:local-ex}
Let $c>0$ and $\tilde{\delta}<\delta_1.$ Let $\boldsymbol{u}_0$ be a $C^2$ function, satisfying $\|\boldsymbol{u}_0\|_{L^\infty(\mathbb{R})}\leq \tilde{\delta }$ and $\|\boldsymbol{u}_0^\prime\|_{L^\infty(\mathbb{R})}\leq c.$ Then there exists a time $T=T(\tilde{\delta},c),$ such that the conservation laws \eqref{CLS} with initial data $\boldsymbol{u}_0$ admits a unique solution $\boldsymbol{u}\in C^2([0,T]\times \mathbb{R}).$ Meanwhile, it holds that
\begin{itemize}
\item all first- and second-order partial derivatives of $\boldsymbol{u}$ are bounded;
\item the solution $\boldsymbol{u}(t,x)$ satisfies $|\boldsymbol{u}(t,x)|<\delta_1$ for all $(t,x)\in [0,T]\times\mathbb{R}$.
\end{itemize}
\end{lemma}

Since characteristic strips $R_i (i=1,2,\cdots,n)$ are independent of source term $\boldsymbol{g}(\boldsymbol{u})$, as in the proof in \cite{John-JDE-1979}, the following non-intersecting property of characteristic strips holds.

\begin{lemma}\label{lem:R-t0}
Under the {\it a priori assumption} \eqref{pri-ass}, there exists a time $t_0$, such that
$t_0\leq C_0 s_0$ and for all $t_0<t\leq T,$
\begin{equation}\label{R-t0}
[\alpha_i(t),\beta_i(t)]\cap [\alpha_j(t),\beta_j(t)]=\varnothing, \quad\quad\quad (i\neq j).
\end{equation}
Here the constant $C_0$ is only depended on $\boldsymbol{A},$ $\delta$ and $\delta_1$.
\end{lemma}

Then we show that any solution of the conservation laws \eqref{CLS} is bounded for short time, which means that the blowup of the solution only occurs for large time.
\begin{lemma}\label{lem:bound-w-t0}
Under the assumptions in Theorem \ref{thmmain} and the {\it a priori assumption} \eqref{pri-ass}, it holds that for all $i\in \{1,2,\cdots,n\},$ 
\begin{equation}\label{bound-w-t0}
|w_i(t,x)|\leq 2W_0, \quad\quad\quad (0\leq t\leq t_0).
\end{equation}
\end{lemma}
\begin{proof}
For all $z\in \mathbb{R},$ denote by
\begin{equation*}
\widetilde{W}(t):=\sup_{i}\ \sup_{z\in\mathbb{R}} |w_i(t,X_i(t,z))|.
\end{equation*}
It is obvious that 
\begin{equation}\label{initial-1}
0\leq |w_i(0,z)|\leq W_0.
\end{equation}
Define
\begin{equation*}
\Gamma=\sup_{|\boldsymbol{u}|\leq \delta_1}\sum_{i,j,k}|\gamma_{ijk}(\boldsymbol{u})|, \quad\quad\quad G=\sup_{|\boldsymbol{u}|\leq \delta_1}\sum_{i,k}|g_{ik}(\boldsymbol{u})|,
\end{equation*}
and notice that for any $j\neq i,$ there exists $y\in\mathbb{R},$ such that
\begin{equation*}
X_i(t,z)=X_j(t,y),
\end{equation*}
which means that
\begin{equation}\label{111}
|w_j(t,X_i(t,z))|=|w_j(t,X_j(t,y))|\leq \widetilde{W}(t).
\end{equation}
Along the characteristic curve $\mathcal{C}_i(z)$, from the differential equation \eqref{eq-w-s} and \eqref{111}, one has
\begin{equation}\label{eq1}
\left|\frac{\mathrm{d}w_i(t,X_i(t,z))}{\mathrm{d}t} \right|\leq \Gamma \widetilde{W}(t)^2+G\widetilde{W}(t).
\end{equation}
Now we solve the differential inequality \eqref{eq1} with initial data \eqref{initial-1}. Since \eqref{eq1} is equivalent to
\begin{equation*}
|w_i(t,x)|\leq W_0+\int_{0}^{t} \left(\Gamma \widetilde{W}(\tau)^2+G\widetilde{W}(\tau)\right) \mathrm{d}\tau,\quad\quad\quad (t,x)\in C_i(z),
\end{equation*}
which holds for any $z\in \mathbb{R}$ and all $i\in\{1,2,\cdots,n\}$, one obtain 
\begin{equation*}
\widetilde{W}(t)\leq W_0+\int_{0}^{t} \left(\Gamma \widetilde{W}(\tau)^2+G\widetilde{W}(\tau)\right) \mathrm{d}\tau.  
\end{equation*}
Solve this integral inequality yields
\begin{equation*}
\widetilde{W}(t)\leq \frac{W_0}{e^{-Gt}-\Gamma G^{-1}\left(1-e^{-Gt}\right)W_0}.
\end{equation*}
The fact that $e^{-x}\geq 1-x$ implies
\begin{align*}
e^{-Gt}-\Gamma G^{-1}\left(1-e^{-Gt}\right)W_0= &e^{-Gt}\left(1+\frac{\Gamma W_0}{G} \right)-\frac{\Gamma W_0}{G}\\
\geq &\left(1-Gt\right) \left(1+\frac{\Gamma W_0}{G} \right)-\frac{\Gamma W_0}{G}\\
=& 1-Gt-\Gamma W_0t.
\end{align*}
By \eqref{small-theta}, \eqref{2}, and the fact that $t_0\leq C_0s_0$ in Lemma \ref{lem:R-t0}, we have
\begin{equation*}
Gt+\Gamma W_0t\leq \frac12,\quad\quad\quad (0\leq t\leq t_0)
\end{equation*}
as long as we choose $\theta_0$ small enough. Therefore,  the lemma holds.
\end{proof}

At the end of this section, we present an estimate for the line integral, which is essentially proved using Green's formula. 
\begin{lemma}[cf. H\"{o}rmander \cite{Hormander}, B\"{a}rlin \cite{Barlin}]\label{lem:dw}
Suppose that $\boldsymbol{u}\in C^2([0,T]\times \mathbb{R})$ is a solution of \eqref{CLS} with $|\boldsymbol{u}(t,x)|\leq \delta_1.$ Let $\tau$ be a $C^1-$arc that crosses the $i$-th family of characteristic lines, and let $A_i(\tau)$ be the surface bounded by the curve segment $\tau$, the two characteristic lines of the $i$-th family passing through the endpoints of $\tau$, and the segment $\tau_0$ on the $x$-axis cut off by these two characteristic lines, as shown in Figure \ref{pic:b}. Then
\begin{equation}\label{Stokes}
\int_{\tau} |w_i(\mathrm{d}x-\lambda_i(\boldsymbol{u})\mathrm{d}t)|\leq \int_{\tau_0}|w_i|\mathrm{d}x+\int_{A_i(\tau)}\left|\sum_{j,k}\Gamma_{ijk}w_jw_k+\sum_{k}g_{ik}w_k\right|\mathrm{d}x\mathrm{d}t.
\end{equation}
Here
\begin{equation*}
\Gamma_{ijk}=\gamma_{ijk}+\delta_{ij}c_{ijk}
\end{equation*}
satisfies
\begin{equation*}
\Gamma_{ijj}=0, \quad\quad\quad \forall i, j\in \{1,2,\cdots,n\}.
\end{equation*}
\end{lemma}

\begin{figure}[htpb]
\centering
\includegraphics[width=0.8\textwidth]{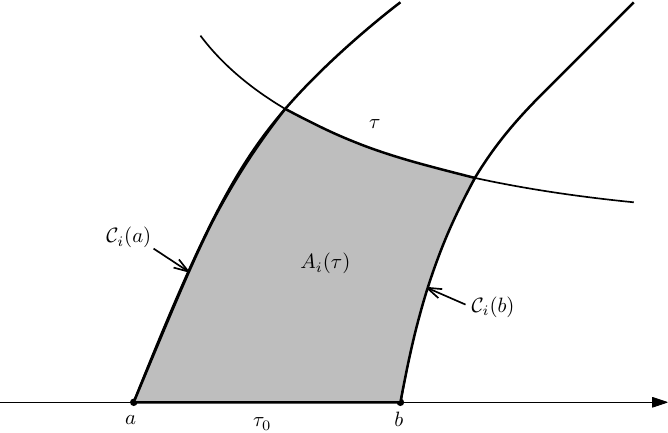}
\caption{the curve $\tau$ and the domain $A_i(\tau)$}
\label{pic:b}
\end{figure}

\section{Proof of Theorem \ref{thmmain}}

\subsection{Some bounded quantities}

To illustrate that the formation of singularities of solutions to \eqref{CLS} is due to a ``gradient catastrophe'', we introduce the following functions:
\begin{align*}
W(t)=&\sup_i\sup_{\substack{(\tau,x)\\0\leq \tau\leq t}}|w_i(\tau,x)|  ,\\
V(t)=&\sup_i\sup_{\substack{(\tau,x)\notin R_i\\0\leq \tau\leq t}}|w_i(\tau,x)| ,\\
U(t)=&\sup_{\substack{(\tau,x)\\0\leq \tau\leq t}}|\boldsymbol{u} (\tau,x)|,\\
G(t)=&\sup_{i,k}\sup_{\substack{(\tau,x)\\0\leq \tau\leq t}}|g_{ik}(\boldsymbol{u} (\tau,x))|,\\
S(t)=&\sup_i\sup_{0\leq \tau\leq t}\left(\beta_i(\tau)-\alpha_i(\tau)\right),\\
J(t)=&\sup_i\sup_{0\leq \tau\leq t}\int_{\alpha_i(\tau)}^{\beta_i(\tau)} |w_i(\tau,x)| \mathrm{d}x.
\end{align*}
Define
\begin{equation}\label{T}
T:=\max_i\frac{17}{\gamma_{iii}(0)W_0^+}.
\end{equation}
In this section, we first show in Proposition \ref{prop:main} that the functions $V(T),$ $U(T),$ $S(T),$ $J(T)$ are bounded, and then prove that $W(t)$ blows up in the interval $t\in [0,T).$

\begin{proposition}\label{prop:main}
Under the assumptions in Theorem \ref{thmmain} and the {\it a priori assumption} \eqref{pri-ass}, for any $t\in [0,T],$ there exist positive constants $C_U$, $C_G$, $C_J$, $C_S$, $C_V$ depending only on $\boldsymbol{A}$, $\boldsymbol{u}_0$ and $\delta_1$, such that
\begin{subequations}\label{bound}
\begin{align}
&U(T)\leq C_U\min\{\theta_1,\theta_0^2 W_0^+\},  \label{bound-U}\\
&G(T)\leq C_G\min\{\theta_1,\theta_0^2 W_0^+\},  \label{bound-G}\\
&J(T)\leq C_J\theta_1,\label{bound-J}\\
&S(T)\leq C_Ss_0,\label{bound-S}\\
&V(T)\leq C_V\theta_1 W_0^+.\label{bound-V}
\end{align}
\end{subequations}
\end{proposition}

In the following, Lemma \ref{lem:U}, Lemma \ref{lem:J}, Lemma \ref{lem:S} and Lemma \ref{lem:V} aim to find the relationship between the functions $U(T)$, $G(T)$, $J(T)$, $S(T)$ and $V(T)$.

\begin{lemma}\label{lem:U}
Under the assumptions in Theorem \ref{thmmain} and the {\it a priori assumption} \eqref{pri-ass}, it holds that
\begin{equation}\label{U}
G(T)\lesssim U(T)\lesssim s_0V(T)+TV(T)+J(T).
\end{equation}
\end{lemma}
\begin{proof}
Since the solutions of hyperbolic conservation laws have finite propagation speed as showed in Lemma \ref{lem:finite-velocity}, for any $t\in [0,T],$ $\boldsymbol{u}(t,x)$ has compact support in $[\alpha_1(t),\beta_n(t)]$. Hence,
\begin{equation}\label{u1}
|\boldsymbol{u} (t,x)|\lesssim \sum_{i}\int_{\alpha_1(t)}^{\beta_n(t)} |w_i(t,x)| \mathrm{d}x.
\end{equation}
From Lemma \ref{lem:finite-velocity}, we have
\begin{equation}\label{u2}
\beta_n(t)-\alpha_1(t)= (\beta_0+\lambda_n(\boldsymbol{0})t)-(\alpha_0+\lambda_1(\boldsymbol{0})t)\leq s_0+(\lambda_n(\boldsymbol{0})-\lambda_1(\boldsymbol{0}))T.
\end{equation}
Now the interval $[\alpha_1(t),\beta_n(t)]$ can be partitioned into two disjoint subsets, $\mathcal{I}(t)$ and $[\alpha_1(t),\beta_n(t)]\backslash \mathcal{I}(t)$, where
\begin{equation*}
\mathcal{I}(t)=\cup_{i=1}^{n}[\alpha_i(t),\beta_i(t)].
\end{equation*}
Through such a decomposition of integration domain, \eqref{u1} can be estimated as
\begin{align*}
|\boldsymbol{u}(t,x)|\lesssim& \sum_{i}\left(\int_{\alpha_i(t)}^{\beta_i(t)} |w_i(t,x)| \mathrm{d}x+\int_{[\alpha_1(t),\beta_n(t)] \backslash\mathcal{I}(t)} |w_i(t,x)| \mathrm{d}x\right)\\
\lesssim & J(T)+(\beta_n(t)-\alpha_1(t))V(T)\\
\lesssim & J(T)+(s_0+T)V(T)
\end{align*}
by \eqref{u2} and the definition of $J(t)$ and $V(t)$. 

Utilizing the definition of $g_{ij}(\boldsymbol{u})$ in Lemma \ref{lem:s-eq-w}, we obtain for any $i,j,$ 
\begin{equation*}
|g_{ij}(\boldsymbol{u})|\lesssim |g(\boldsymbol{u})|+|\nabla_{\boldsymbol{u}}g(\boldsymbol{u})|\lesssim |\boldsymbol{u}|\lesssim U(T).
\end{equation*}
Hence, it holds that
\begin{equation}\label{G}
G(T)\lesssim U(T).
\end{equation}
This finishes the proof of the lemma.
\end{proof}

\begin{lemma}\label{lem:J}
Under the assumptions in Theorem \ref{thmmain} and the {\it a priori assumption} \eqref{pri-ass}, it holds that
\begin{equation}\label{J}
 J(T)\lesssim s_0^2W_0+s_0^2W_0^2 +TJ(T)(V(T)+G(T))+TS(T)V(T)(V(T)+G(T)).
\end{equation}
\end{lemma}
\begin{proof}
Applying Lemma \ref{lem:dw} to the line segment $[\alpha_i(t),\beta_i(t)]$
across the characteristic strip $R_i$,  we therefore have
\begin{multline}\label{j1}
\int_{\alpha_i(t)}^{\beta_i(t)} |w_i(t,x)| \mathrm{d}x\\
 \leq \int_{\alpha_0}^{\beta_0} |w_i(0,x)| \mathrm{d}x +\int_{0}^{t}\int_{\alpha_i(\tau)}^{\beta_i(\tau)} \left|\sum_{j,k}\Gamma_{ijk}w_jw_k+\sum_{k}g_{ik}w_k\right| \mathrm{d}x\mathrm{d}\tau.
\end{multline}
Notice that
\begin{equation*}
|\boldsymbol{u}_0^\prime|\lesssim W_0,\quad\quad\quad |\boldsymbol{u}_0|\lesssim s_0W_0,
\end{equation*}
\begin{equation*}
|\boldsymbol{l}_i(\boldsymbol{u}_0)-\boldsymbol{l}_i(\boldsymbol{0})|\lesssim s_0W_0.
\end{equation*}
Hence, the first term on the right-hand side of \eqref{j1} satisfies
\begin{equation}\label{j2}
\left|\int_{\alpha_0}^{\beta_0} w_i(0,x) \mathrm{d}x\right|=\left|\int_{\alpha_0}^{\beta_0} \left(\boldsymbol{l}_i(\boldsymbol{u}_0)-\boldsymbol{l}_i(\boldsymbol{0}) \right)\cdot \boldsymbol{u}_0^\prime \mathrm{d}x\right|\lesssim s_0^2W_0^2.
\end{equation}
 
For $0\leq t\leq t_0,$ the fact that $t_0\lesssim s_0,$ \eqref{u2} and the Lemma \ref{lem:bound-w-t0} imply
\begin{equation}\label{j3}
\int_{\alpha_i(t)}^{\beta_i(t)} |w_i(t,x)| \mathrm{d}x\lesssim { s_0^2W_0+s_0^2W_0^2}.
\end{equation}
For $t_0\leq t\leq T,$ from Lemma \ref{lem:R-t0} we know that for all $j\neq i,$ 
\begin{equation*}
[\alpha_i(t),\beta_i(t)]\cap [\alpha_j(t),\beta_j(t)]= \varnothing.
\end{equation*}
Hence, for any $t\in [t_0,T],$ we have 
\begin{equation}\label{j4}
|w_j(t,x)|\leq V(t),\quad\quad\quad \forall x\in [\alpha_i(t),\beta_i(t)].
\end{equation}
Then, by \eqref{j3} and the fact that $\Gamma_{ijj}=0$ in Lemma \ref{lem:dw}, the second term on the right-hand side of inequality \eqref{j1} can be separated as 
\begin{align*}
&\int_{0}^{t}\int_{\alpha_i(\tau)}^{\beta_i(\tau)} \left|\sum_{j,k}\Gamma_{ijk}w_jw_k+\sum_{k}g_{ik}w_k\right| \mathrm{d}x\mathrm{d}\tau\\
=&\left(\int_{0}^{t_0}\int_{\alpha_i(\tau)}^{\beta_i(\tau)}+\int_{t_0}^{t}\int_{\alpha_i(\tau)}^{\beta_i(\tau)}\right) \left|\sum_{j,k}\Gamma_{ijk}w_jw_k+\sum_{k}g_{ik}w_k\right| \mathrm{d}x\mathrm{d}\tau\\
\lesssim & s_0^2W_0+s_0^2W_0^2+
\sum_{j\neq i}\int_{t_0}^{t}\int_{\alpha_i(\tau)}^{\beta_i(\tau)}|w_iw_j|\mathrm{d}x\mathrm{d}\tau+\sum_{j\neq i}\int_{t_0}^{t}\int_{\alpha_i(\tau)}^{\beta_i(\tau)}w_j^2\mathrm{d}x\mathrm{d}\tau\\
&+G(T)\int_{t_0}^{t}\int_{\alpha_i(\tau)}^{\beta_i(\tau)}|w_i|\mathrm{d}x\mathrm{d}\tau+G(T)\sum_{j\neq i}\int_{t_0}^{t}\int_{\alpha_i(\tau)}^{\beta_i(\tau)}|w_j|\mathrm{d}x\mathrm{d}\tau\\
\lesssim & s_0^2W_0+s_0^2W_0^2+TJ(T)V(T)+TS(T)V(T)^2+TG(T)J(T)+TG(T)S(T)V(T).
\end{align*}
Substituting the above estimate into \eqref{j1} yields \eqref{J}.
\end{proof}

\begin{lemma}\label{lem:S}
Under the assumptions in Theorem \ref{thmmain} and the {\it a priori assumption} \eqref{pri-ass}, it holds that
\begin{equation}\label{S}
S(T)\lesssim s_0+TS(T)V(T)+TJ(T).
\end{equation}
\end{lemma}
\begin{proof}
For $0\leq t\leq t_0$, it is obvious that
\begin{align*}
\left|\frac{\mathrm{d}(\beta_i(t)-\alpha_i(t))}{\mathrm{d}t} \right|\leq&|\lambda_i(\boldsymbol{u}(t,\beta_i(t)))-\lambda_i(\boldsymbol{u}(t,\alpha_i(t)))| \\
\lesssim & |\boldsymbol{u}(t,\beta_i(t))-\boldsymbol{u}(t,\alpha_i(t))|\\
\lesssim & W_0(\beta_i(t)-\alpha_i(t)).
\end{align*}
Since $t_0\lesssim s_0$, for $0\leq t\leq t_0$,
\begin{equation} S(t)\lesssim s_0+t_0W_0S(t), \end{equation}
which, combined with the smallness of $s_0W_0$, implies
\begin{equation*}
S(t)\lesssim s_0,\quad\quad\quad \forall t\in[0,t_0].
\end{equation*}

For $t_0< t\leq T$, by \eqref{nabla-lambda} and the definition of $\alpha_i(t)$ and $\beta_i(t),$ it holds that
\begin{equation*}
\frac{\mathrm{d}(\beta_i(t)-\alpha_i(t))}{\mathrm{d}t}=\int_{\alpha_i(t)}^{\beta_i(t)} \frac{\partial}{\partial x} \lambda_i(\boldsymbol{u}(t,x)) \mathrm{d}x =\int_{\alpha_i(t)}^{\beta_i(t)} \sum_{k}c_{iik}(\boldsymbol{u})w_k(\boldsymbol{u}) \mathrm{d}x.
\end{equation*}
Notice that the integral interval is $[\alpha_i(t), \beta_i(t)].$ Then we have
\begin{equation*}
\left|\frac{\mathrm{d}(\beta_i-\alpha_i)}{\mathrm{d}t} \right|\lesssim S(t)V(t)+J(t)
\end{equation*}
by using \eqref{j4}. The lemma is proved by integrating the above inequality over $[t_0,T]$.
\end{proof}

\begin{lemma}\label{lem:V}
Under the assumptions in Theorem \ref{thmmain} and the {\it a priori assumption} \eqref{pri-ass}, it holds that
\begin{multline}\label{V}
V(T)\lesssim s_0W_0+s_0W_0^2
+TV(T)(V(T)+G(T))\\+TJ(T)(V(T)+G(T))^2+TS(T)V(T)(V(T)+G(T))^2.
\end{multline}
\end{lemma}
\begin{proof}
Fix $i\in \{1,2,\cdots,n\}.$ From the definition of $V(t)$, it suffices to estimate $w_i(t,X_i(t,z))$ for $z \notin I_0$.

Fix $z \notin I_0.$ Note that the characteristic curve $\mathcal{C}_i(z) = (t, X_i(t, z))$ may belong to different characteristic strip $R_m$ $(m\in \{1,2,\cdots,n\})$ as $t$ varies (see Figure \ref{pic:2}). Therefore, we partition the time interval $(t_0, T]$. Define
\begin{equation*}
\omega_m:=\{t_0<\tau\leq T:(\tau,X_i(\tau,z))\in R_m\},
\end{equation*}
\begin{equation*}
\omega:=\cup_{m=1}^n\omega_m,\quad\quad\quad \omega^c:=(t_0,T]\backslash \omega.
\end{equation*}
Then it is easy to see that $\omega_i$ is the empty set. And from Lemma \ref{lem:R-t0} we know that 
$$\omega_j\cap\omega_k=\varnothing, \quad\quad\quad j\neq k.$$
\begin{figure}[htpb]
\centering
\includegraphics[width=\textwidth]{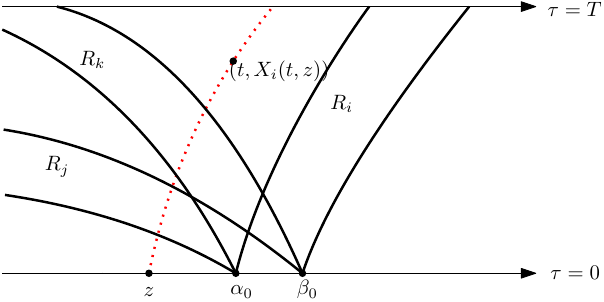}
\caption{ $\mathcal{C}_i(z)=(t,X_i(t,z))$ belonging to different  $R_m$}
\label{pic:2}
\end{figure}

As $w_i(0,z)=0,$ along the characteristic curve $\mathcal{C}_i(z)$, Lemma \ref{lem:s-eq-w} implies
\begin{equation*}
w_i(t,X_i(t,z))=\int_{0}^{t} \bigg(\sum_{j,k}\gamma_{ijk}w_jw_k+\sum_kg_{ik}w_k\bigg)(\tau,X_i(\tau,z))\mathrm{d}\tau.  
\end{equation*}
We split the integral interval into three parts:
\begin{align}
&|w_i(t,X_i(t,z))|\nonumber\\
\lesssim & \left(\int_{0}^{t_0} +\int_{\omega^c}+\sum_{m\neq i}\int_{\omega_m}\right) \bigg(\sum_{j,k}\gamma_{ijk}w_jw_k+\sum_kg_{ik}w_k\bigg)(\tau,X_i(\tau,z))\mathrm{d}\tau  \nonumber  \\
:=& \mathcal{M}_1+\mathcal{M}_2+\mathcal{M}_3. \label{M}
\end{align}

Now we estimate each term on the right-hand side of the above inequality one by one.
For $\tau\in [0,t_0]$, from Lemma \ref{lem:bound-w-t0} and $t_0\lesssim s_0$ we have
\begin{equation}\label{M1}
\mathcal{M}_1\lesssim s_0W_0^2+s_0W_0.
\end{equation}
For $\tau\in \omega^c$, from the definition of $V$ we have
\begin{equation}\label{M2}
\mathcal{M}_2\lesssim V^2T+VT.
\end{equation}
For $\tau \in \omega$, notice that $\gamma_{ijj}=0$ if $j\neq i$. Hence
\begin{align*}
\mathcal{M}_3=& \sum_{m\neq i}\int_{\omega_m} \bigg(\sum_{j,k}\gamma_{ijk}w_jw_k+\sum_kg_{ik}w_k\bigg)(\tau,X_i(\tau,z))\mathrm{d}\tau\\
\lesssim &  \sum_{m\neq i}\int_{\omega_m} \left(\sum_{j\neq m}|w_m w_j|+\sum_{j\neq m} w_j^2+G(T)|w_m|+G(T)\sum_{j\neq m}|w_j|\right)(\tau,X_i(\tau,z))\mathrm{d}\tau\\
\lesssim & (V(T)+G(T))\sum_{m\neq i}\underbrace{{\int_{\omega_m} |w_m(\tau,X_i(\tau,z))|\mathrm{d}\tau}}_{:=\mathcal{I}_m}+TV(T)(V(T)+G(T)).
\end{align*}
Here we have used the fact that
\begin{equation*}
|w_j(\tau,X_i(\tau,z)| \lesssim  V(T),\quad\quad\quad \tau\in \omega_m, \ j\neq m.
\end{equation*}

Now we focus on the estimate of $\mathcal{I}_m.$ Define the curve (see Figure \ref{pic:a})
\begin{equation*}
\tau_m:=\{(\tau,X_i(\tau,z)):\tau\in \omega_m\}.
\end{equation*}
\begin{figure}[htpb]
\centering
\includegraphics[width=0.7\textwidth]{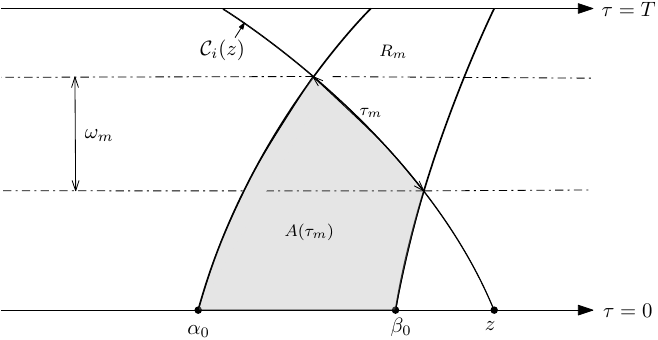}
\caption{The curve $\tau_m$ and the domain $A(\tau_m)$}
\label{pic:a}
\end{figure}
Since $\boldsymbol{A}(\boldsymbol{0})$ is strictly hyperbolic, it holds that when $|\boldsymbol{u}|\leq \delta_1$, $|\lambda_i-\lambda_m|$ has positive lower bound for all $m\neq i.$ Using Lemma \ref{lem:dw}, we have
\begin{align*}
\mathcal{I}_m\lesssim & \int_{\omega_m} |w_m(\lambda_i-\lambda_m)|(\tau,X_i(\tau,z))\mathrm{d}\tau \\
= &\int_{\tau_m} |w_m(\mathrm{d}x-\lambda_m\mathrm{d}\tau)|\\
\lesssim& \int_{\alpha_0}^{\beta_0} |w_m(0,x)| \mathrm{d}x +\int_{0}^{T}\int_{\alpha_m(\tau)}^{\beta_m(\tau) } \left|\sum_{j,k}\Gamma_{mjk}w_jw_k+\sum_{k}g_{mk}w_k\right| \mathrm{d}x\mathrm{d}\tau 
\end{align*}
One can find that the estimate of  $ \mathcal{I}_m$ is similar to that of $J(T)$ in \eqref{J}. Therefore, from \eqref{J} and \eqref{j2}, we have
\begin{equation*}
 \mathcal{I}_m\lesssim s_0^2W_0+s_0^2W_0^2 +TJ(T)(V(T)+G(T))+TS(T)V(T)(V(T)+G(T)).
\end{equation*}
Substituting the above inequality into the estimate of $\mathcal{M}_3,$ we have
\begin{multline}\label{M3}
\mathcal{M}_3\lesssim (V(T)+G(T))\left[s_0^2W_0+s_0^2W_0^2 +TJ(T)(V(T)+G(T))\right.\\
\left.+TS(T)V(T)(V(T)+G(T))\right]+TV(T)(V(T)+G(T)).
\end{multline}
The lemma is proved by substituting \eqref{M1}, \eqref{M2} and  \eqref{M3} into \eqref{M}, and using \eqref{small-theta} and $s_0W_0\leq \theta_2.$
\end{proof}



With the above results in hand, we close this subsection by giving the proof of  Proposition \ref{prop:main}.
\begin{proof}[proof of the Proposition \ref{prop:main}]
To simplify notation, we write $V(T), U(T), G(T), S(T)$ and $J(T)$ as $V,U,G,S$ and $J$. From Lemma \ref{lem:U} to Lemma \ref{lem:V}, we obtain
\begin{align}
G&\lesssim U\lesssim s_0V+TV+J,\label{1U}\\
J&\lesssim s_0(s_0W_0+s_0W_0^2)+TJ(V+G)+TSV(V+G),  \label{1J}\\
S&\lesssim s_0+TSV+TJ, \label{1S}\\
V&\lesssim s_0W_0+s_0W_0^2+TV(V+G)+T(J+SV)(V^2+G^2). \label{1V}
\end{align}

The main idea is: we first present a {\it a priori assumption} (H)
\begin{align}
TU&\leq \frac{\sqrt{s_0W_0+s_0W_0^2}}{W_0^+} , \tag{H1}\\
TV&\leq \sqrt{\frac{s_0W_0+s_0W_0^2}{W_0^+ }} , \tag{H2} \\
s_0V&\leq\sqrt{s_0(s_0W_0+s_0W_0^2)}, \tag{H3}\\
J&\leq \sqrt{\frac{s_0W_0+s_0W_0^2}{W_0^+} }, \tag{H4} 
\end{align}
then prove a stronger priori estimate (C)
\begin{align}
TU&\lesssim \frac{s_0W_0+s_0W_0^2}{(W_0^+)^2} , \tag{C1}\\
TV&\lesssim \frac{s_0W_0+s_0W_0^2}{W_0^+ } , \tag{C2} \\
s_0V&\lesssim s_0(s_0W_0+s_0W_0^2), \tag{C3}\\
J&\lesssim \frac{s_0W_0+s_0W_0^2}{W_0^+}. \tag{C4} 
\end{align}
Here all the constants appearing on the right-hand side are small by \eqref{small-theta} and Lemma \ref{lem:sW+sW2}.

By the smallness of $TV$ and $TU,$ it can be derived from \eqref{1J} that
\begin{equation}\label{j-}
J\lesssim s_0(s_0W_0+s_0W_0^2)+(TV+TU)\cdot SV.
\end{equation}
Furthermore, by the definition of $T$ in \eqref{T},
\begin{equation*}
TJ\lesssim \frac{s_0W_0+s_0W_0^2}{W_0^+}s_0+(TV+TU)\cdot TV\cdot S.
\end{equation*}
Substituting the above estimate into \eqref{1S} and using the smallness of $TV,$ $TU$ and $\frac{s_0W_0+s_0W_0^2}{W_0^+}$ yields
\begin{equation}\label{s}
S\lesssim s_0.
\end{equation}
Now we return to estimate $J$. From \eqref{2} and the lower bound of $|\boldsymbol{u} ^\prime_0(t,x)|$, we have
\begin{equation}\label{12}
s_0\leq \frac{\theta_2}{W_0^+},\quad\quad\quad s_0\leq \frac{\theta_2}{W_0}\lesssim \theta_2.
\end{equation}
Then, using \eqref{s}, \eqref{j-} can be rewritten as
\begin{align*}
J&\lesssim s_0(s_0W_0+s_0W_0^2)+(TV+TU)\cdot s_0V \\
&\lesssim \left(\theta_2+\sqrt{s_0W_0^+}+\sqrt{s_0} \right) \frac{s_0W_0+s_0W_0^2}{W_0^+}\\
&\lesssim \sqrt{\theta_2}\frac{s_0W_0+s_0W_0^2}{W_0^+}
\end{align*}
by using (H1), (H2) and (H3). Thus, (C4) holds.

Next, we estimate $V.$ From \eqref{s}, \eqref{1V} can be rewritten as
\begin{equation*}
V\lesssim s_0W_0+s_0W_0^2+(TV+TU+TV\cdot J+TV\cdot s_0V)\cdot V+(J+s_0V)TU^2.
\end{equation*}
Using the smallness of $TV,$ $TU,$ $J$ and  $s_0V$, we have
\begin{equation}\label{v-}
V\lesssim s_0W_0+s_0W_0^2+(J+s_0V)TU^2.
\end{equation}
From \eqref{12} we have $s_0\lesssim T.$ 
Substituting (C4) and \eqref{v-} into \eqref{1U} yields
\begin{align*}
G\lesssim U&\lesssim TV+J \\
&\lesssim \frac{s_0W_0+s_0W_0^2}{W_0^+}+[(TJ+s_0TV)\cdot TU]\cdot U
\end{align*}
Thus, we have
\begin{equation}\label{g-}
G\lesssim U\lesssim \frac{s_0W_0+s_0W_0^2}{W_0^+}
\end{equation}
by the smallness of $TJ$, $s_0$, $TV$ and $TU$. This means that (C1) holds.

It remains to prove (C2) and  (C3). By the smallness of $s_0TU^2,$ actually,
\begin{equation*}
s_0TU^2\lesssim (TU)^2\lesssim1,
\end{equation*}
 \eqref{v-} can be rewritten as
\begin{equation*}
V\lesssim s_0W_0+s_0W_0^2+JTU^2.
\end{equation*}
Now substituting (C4) and \eqref{g-} into the above inequality yields
\begin{equation}\label{v}
V\lesssim s_0W_0+s_0W_0^2,
\end{equation}
from which (C2) and (C3) holds.

Consequently, \eqref{bound-U}--\eqref{bound-V} hold by \eqref{s}, \eqref{g-}, \eqref{v}, (C4) and Lemma \ref{lem:sW+sW2}.
\end{proof}

\subsection{Blow-up of the solution} 
By the definition of $W_0^+,$  there exists a certain $i\in\{1,2,\cdots,n\}$  and $z_0\in I_0,$ such that $w_i(0,z_0)=W_0^+>0.$ Define
\begin{equation*}
\mathcal{W}(t)=w_i(t,X_i(t,z_0)),
\end{equation*}
then $\mathcal{W}(0)=W_0^+>0.$

For $0\leq t\leq t_0$, we have $|\mathcal{W}^\prime(t)|\lesssim W_0^2+W_0$ from Lemma \ref{lem:s-eq-w} and Lemma \ref{lem:bound-w-t0}. Hence, 
\begin{equation*}
\mathcal{W}(t_0)-W_0^+\geq -s_0(W_0^2+W_0)\geq -\theta_1W_0^+
\end{equation*}
by Lemma \ref{lem:sW+sW2}. Taking $\theta_1\leq\theta_0$ small enough, we have
\begin{equation}\label{Wt0}
\mathcal{W}(t_0)\geq \frac12 W_0^+.
\end{equation}
 
For $t_0\leq t\leq T$, we have
\begin{equation*}
\mathcal{W}^\prime(t)\geq \frac12\gamma_{iii}(\boldsymbol{0})\mathcal{W}^2-\Gamma\left(|\mathcal{W}|V+V^2\right)-G\left(|\mathcal{W}|+V\right)
\end{equation*}
from \eqref{g>0} and Lemma \ref{lem:s-eq-w}.
We firstly give the following {\it a priori assumption}: for any $t_0\leq t\leq T,$  
\begin{equation}\label{APA}
 \left\{\begin{aligned}
&\mathcal{W}(t)\geq \frac12\mathcal{W}(t_0),   \\
&V(t)\leq 4C_V\theta_1 \mathcal{W}(t). 
\end{aligned}\right.
\end{equation}
On one hand, under this assumption, one can get
\begin{align*}
\mathcal{W}^\prime(t)\geq&\frac12\gamma_{iii}(\boldsymbol{0})\mathcal{W}^2-\Gamma(\mathcal{W}V+V^2)-G(\mathcal{W}+V) \\
\geq&\left(\frac12\gamma_{iii}(\boldsymbol{0})-4C_V\theta_1(1+4C_V\theta_1)\Gamma\right)\mathcal{W}^2-G(1+4C_V\theta_1)\mathcal{W}.
\end{align*}
Taking $\theta_0$ small enough, using \eqref{Wt0}, \eqref{APA} and Proposition \ref{prop:main}, it holds that
\begin{align*}
\mathcal{W}^\prime(t) \geq &\frac14\gamma_{iii}(\boldsymbol{0})\mathcal{W}^2-2G\mathcal{W}\\
\geq&\left(\frac{1}{16}\gamma_{iii}(\boldsymbol{0})W_0^+-2G\right)\mathcal{W}\\
\geq& \left(\frac{1}{16}\gamma_{iii}(\boldsymbol{0})-2C_G\theta_0^2\right)W_0^+\mathcal{W}\\
\geq&\frac1{32}\gamma_{iii}(\boldsymbol{0})W_0^+\mathcal{W},
\end{align*}
which yields 
\begin{equation*}
\mathcal{W}(t)\geq \mathcal{W}(t_0)e^{\frac1{32}\gamma_{iii}(\boldsymbol{0})W_0^+(t-t_0)}\geq \mathcal{W}(t_0)\geq \frac12 W_0^+>0.
\end{equation*}
This means $\mathcal{W}^\prime(t)>0,$ in other words, 
\begin{equation}\label{APE1}
\mathcal{W}(t)\geq \mathcal{W}(t_0),\quad\quad\quad t_0\leq t\leq T.
\end{equation}
On the other hand, by \eqref{bound-V} and \eqref{Wt0},  we obtain 
\begin{equation*}
V(t)\leq C_V\theta_1 W_0^+\leq 2C_V\theta_1\mathcal{W}(t_0).
\end{equation*}
Using the monotonicity of $\mathcal{W}(t)$ in \eqref{APE1}, the above estimate can be rewritten as
\begin{equation}\label{APE2}
V(t)\leq 2C_V\theta_1\mathcal{W}(t),\quad\quad\quad t_0\leq t\leq T.
\end{equation}
Since the estimates \eqref{APE1} and \eqref{APE2} are stronger than the {\it a priori assumption} \eqref{APA}, we get \eqref{APE1} and \eqref{APE2} always holds by bootstrap principle.

Having the estimates \eqref{APE1} and \eqref{APE2} in hand, we now start to prove the blowup of the function $\mathcal{W}(t).$  Notice that the function $\mathcal{W}(t)$ satisfies the following differential inequality
\begin{equation*}
 \left\{\begin{aligned}
&\mathcal{W}^\prime(t)\geq  \frac14\gamma_{iii}(\boldsymbol{0})\mathcal{W}^2-2G\mathcal{W}, \quad\quad\quad t_0\leq t\leq T, \\
&  \mathcal{W}(t_0)\geq\frac12 W_0^+.
\end{aligned}\right.
\end{equation*}
One the one hand, the solution of the ordinary differential equation 
\begin{equation*}
 \left\{\begin{aligned}
&y^\prime(t)=\alpha y(t)^2-\beta y(t), \\
&y(0)=y_0>0
\end{aligned}\right.
\end{equation*}
will blow up if and only if $\alpha y_0>\beta.$ The solution of the above ODE is 
\begin{equation*}
y(t)=\frac{\beta}{\alpha-e^{\beta t}\left(\alpha-\frac{\beta}{y_0} \right) }, \quad\quad\quad 0\leq t<t_{max},
\end{equation*}
where 
\begin{equation}\label{t-max}
t_{\max}=-\frac{1}{\beta}\ln\left(1-\frac{\beta}{\alpha y_0} \right)
\end{equation}
is the life span of $y(t).$ Now taking $\alpha=\frac14\gamma_{iii}(\boldsymbol{0}),$ $\beta=2G,$ and $y_0=\frac12 W_0^+,$ then
\begin{equation}\label{1}
\frac{\beta}{\alpha y_0}=\frac{16G}{\gamma_{iii}(\boldsymbol{0})W_0^+}\leq \frac{16C_G}{\gamma_{iii}(\boldsymbol{0})}\theta_0^2<\frac12,
\end{equation}
by using Proposition \ref{prop:main} and taking $\theta_0$ small enough. Taking the bound of $\frac{\beta}{\alpha y_0}<\frac12$ and the inequality
\begin{equation*}
-\ln (1-x)< x+2x^2,\quad\quad\quad 0< x< \frac12
\end{equation*}
into account, one obtain from \eqref{t-max} and \eqref{1} that 
\begin{align*}
 t_{max}<&\frac{1}{\beta}\left(\frac{\beta}{\alpha y_0}+\frac{2\beta^2}{\alpha^2 y_0^2}\right)  \\
= & \left(1+\frac{2\beta}{\alpha y_0}  \right)\frac{1}{\alpha y_0}\\
< &   \frac{2}{\alpha y_0}=\frac{16}{ \gamma_{iii}(\boldsymbol{0})W_0^+} .
\end{align*}
On the other hand, by comparison principle, we have
\begin{equation*}
\mathcal{W}(t)\left\{\begin{aligned}
&\in \left[\frac12 W_0^+,\frac32 W_0^+\right], \quad\quad\quad &&0\leq t\leq t_0, \\
&\geq \frac{\beta}{\alpha-e^{\beta (t-t_0)}\left(\alpha-\frac{\beta}{y_0} \right) }, \quad\quad\quad &&t_0<t<t_0+t_{max}.
\end{aligned}\right.
\end{equation*}
 The life span of $\mathcal{W}(t)$ satisfies
\begin{equation*}
t_0+t_{max}<\frac{17}{ \gamma_{iii}(\boldsymbol{0})W_0^+}=:T, 
\end{equation*}
by noticing that
\begin{equation*}
t_0\lesssim s_0\leq \theta_0/W_0\leq\theta_0/W_0^+
\end{equation*}
and taking $\theta_0$ small enough.

\section{Proof of Theorem \ref{thm:2}}

This section is dedicated to the proof of Theorem \ref{thm:2}. Since the initial data have compact support, we first present a corresponding result at the beginning of this section.

\begin{lemma}\label{lem:f}
Given $\alpha_0<\beta_0$ and let $s_0=\beta_0-\alpha_0.$ Then for any smooth function $f(x)$ that is not identically zero and has compact support in $[\alpha_0,\beta_0],$ we have that
\begin{equation}\label{f}
|f(x_0)|< s_0 |f^\prime(x_0)|,
\end{equation}
where $x_0$ is a global minimum point of $f^\prime(x).$ The constant $s_0$ in the inequality \eqref{f} is the sharp (or optimal) constant, i.e., for any $0< \delta<\frac{s_0}{4},$ there exists a non-identically-zero function $f_\delta(x)$ with compact support in the interval $[\alpha_0,\beta_0]$ that satisfies 
\begin{equation}\label{conter-f}
|f_\delta(x_0)|\geq(s_0-4\delta)|f_\delta^\prime(x_0)|.
\end{equation}
\end{lemma}
 
\begin{proof}
We firstly prove \eqref{f}. Notice that 
$$f^\prime(x)\geq f^\prime(x_0).$$
It holds that
\begin{equation*}
\alpha_0< x_0< \beta_0,
\end{equation*}
\begin{equation*}
f^\prime(x_0)<0
\end{equation*}
and
\begin{align*}
f(x_0)=&\int_{\alpha_0}^{x_0} f^\prime(y) \mathrm{d}y > (x_0-\alpha_0)f^\prime(x_0), \\
-f(x_0)=&\int_{x_0}^{\beta_0} f^\prime(y) \mathrm{d}y > (\beta_0-x_0)f^\prime(x_0)
\end{align*}
by the fact that $f(x)$ is not identically zero and has compact support in $[\alpha_0,\beta_0].$ Then, we have
\begin{equation*}
(x_0-\alpha_0)f^\prime(x_0)< f(x_0)< - (\beta_0-x_0)f^\prime(x_0),
\end{equation*}
which implies \eqref{f} holds.

Secondly, we show that the constant $s_0$ in the inequality \eqref{f} is the sharp constant. The main idea is: first construct a function  $\tilde{f}_\delta\in C^1([\alpha_0+\delta,\beta_0-\delta])$ satisfying \eqref{conter-f}, and then extend it to the whole space $\mathbb{R}$ and smooth it (e.g., by convolution) to obtain a smooth function $f_\delta(x).$

\begin{figure}[htpb]
\centering
\includegraphics[width=0.8\columnwidth]{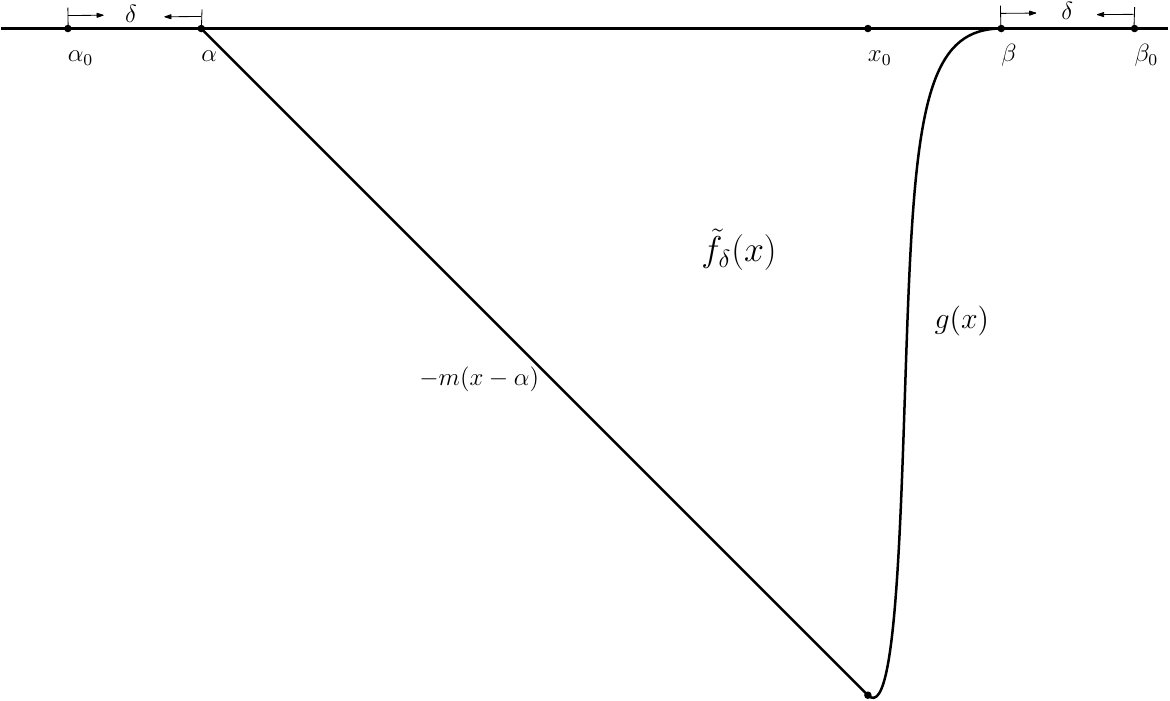} %
\caption{Image of $\tilde{f}_\delta(x)$}
\label{pic-f}
\end{figure}
Let $\alpha=\alpha_0+\delta$ and $\beta=\beta_0-\delta.$ Given constants $m>0$ and $x_0\in(\alpha,\beta)$ to be determined, the  function  $\tilde{f}_\delta$ is constructed as a piecewise function (see Figure \ref{pic-f}):
\begin{equation*}
\tilde{f}_\delta(x)=\left\{\begin{aligned}
&0,&&x\leq \alpha,\\
&-m(x-\alpha), && \alpha\leq x\leq x_0,  \\
&g(x), && x_0\leq x\leq \beta,\\
&0,&&x\geq \beta,
\end{aligned}\right.
\end{equation*}
where 
\begin{equation*}
g(x)=A(x-\beta)^3+B(x-\beta)^2,\quad\quad\quad (x_0\leq x\leq \beta)
\end{equation*}
and the constants 
\begin{align*}
A=&-m\frac{2(x_0-\alpha)+(\beta-x_0)}{(\beta-x_0)^3}<0, \\
B=&-m\frac{3(x_0-\alpha)+(\beta-x_0)}{(\beta-x_0)^2}<0
\end{align*}
are chosen such that the piecewise function $\tilde{f}_\delta\in C^1([\alpha,\beta]),$ i.e., 
\begin{align*}
-m(x_0-\alpha)=&g(x_0), \\
-m=&g^\prime(x_0).
\end{align*}
Now we prove that $x_0$ is the global minimum point of function $\tilde{f}^\prime_\delta(x),$ i.e.,
\begin{equation*}
\tilde{f}_\delta^\prime(x_0)=\inf_{\alpha\leq x\leq \beta}f_\delta^\prime(x).
\end{equation*}
To do this, by the fact that
\begin{equation*}
\tilde{f}^\prime_\delta(x)=-m,\quad\quad\quad x\in [\alpha,x_0],
\end{equation*}
it suffices to prove that
\begin{equation}\label{g}
\tilde{f}^\prime_\delta(x)=g^\prime(x)\geq -m\quad\quad\quad x\in [x_0,\beta].
\end{equation}
Notice that 
\begin{equation*}
g^\prime(x_0)=-m,\quad\quad\quad g^\prime(\beta)=0,
\end{equation*}
\begin{equation*}
g^{\prime\prime}(x_0)=6A(x_0-\beta)+2B=m\frac{6(x_0-\alpha)+4(\beta-x_0)}{(\beta-x_0)^2}>0,
\end{equation*}
\begin{equation*}
g^{\prime\prime}(\beta)=2B<0,
\end{equation*}
and 
\begin{equation*}
g^{\prime\prime\prime}(x)=6A<0.
\end{equation*}
Hence, the function $g^{\prime\prime}(x)$ has a zero at $x_1\in (x_0,\beta).$ Meanwhile, $g^\prime(x)$ is monotonically increasing on the interval $[x_0,x_1]$ and  monotonically decreasing on the interval $[x_1,\beta].$ These observations show that
\begin{equation*}
g^\prime(x)\geq g^\prime(x_0)=-m,\quad\quad\quad x\in [x_0,x_1],
\end{equation*}
and
\begin{equation*}
g^\prime(x)\geq g^\prime(\beta)=0,\quad\quad\quad x\in [x_1,\beta],
\end{equation*}
which imply that \eqref{g} holds. Hence, $x_0$ is indeed a global minimum point of function $\tilde{f}^\prime_\delta(x).$

It is easy to see that
\begin{equation}\label{tem-1}
\tilde{f}_\delta(x_0)=-m(x_0-\alpha),\quad\quad\quad \tilde{f}^\prime_\delta(x_0)=-m.
\end{equation}
Hence, if we choose $x_0=\beta_0-2\delta=\beta-\delta,$ one has
\begin{equation}\label{quotient}
\frac{|\tilde{f}_\delta(x_0)|}{|\tilde{f}^\prime_\delta(x_0)| }=s_0-3\delta.
\end{equation}

Now we smooth the function $\tilde{f}_\delta(x).$ Take the mollifier $\varphi_\varepsilon\in C_c^\infty(\mathbb{R}),$ with support in $[-\varepsilon,\varepsilon],$ satisfying $\varphi_\varepsilon\geq0$ and $\int_{\mathbb{R}}^{} \varphi_\varepsilon(x) \mathrm{d}x =1.$ Take $\varepsilon<\delta,$ and define 
\begin{equation*}
f_\delta(x)=(\tilde{f}_\delta\ast \varphi_\varepsilon)(x)=\int_{\mathbb{R}}^{} \tilde{f}_\delta(x-y)\varphi_\varepsilon(y) \mathrm{d}y.  
\end{equation*}
Then the support of $f_\delta$ is in 
\begin{equation*}
[\alpha-\varepsilon,\beta+\varepsilon]=[\alpha_0+\delta-\varepsilon,\beta_0-\delta+\varepsilon]\subset[\alpha_0,\beta_0].
\end{equation*}
Since 
\begin{align*}
\lim\limits_{\varepsilon\to0^+}|f_\delta(x_0)-\tilde{f}_\delta(x_0)|=&0, \\
\lim\limits_{\varepsilon\to0^+}|f^\prime_\delta(x_0)-\tilde{f}^\prime_\delta(x_0)|=&0,
\end{align*}
we have from \eqref{tem-1} and \eqref{quotient} that
\begin{equation*}
\frac{|f_\delta(x_0)|}{|f^\prime_\delta(x_0)|}=\frac{|(f_\delta(x_0)-\tilde{f}_\delta(x_0))+\tilde{f}_\delta(x_0)|}{|(f^\prime_\delta(x_0)-\tilde{f}^\prime_\delta(x_0))+\tilde{f}^\prime_\delta(x_0)|}\geq s_0-4\delta 
\end{equation*}
as long as $\varepsilon$ is small enough. Thus, we construct a counterexample function $f_\delta(x)$ that satisfies inequality \eqref{conter-f} and whose derivative function $f^\prime_\delta(x)$ has a global minimum point near $x_0$. This demonstrates that $s_0$ in \eqref{f} is indeed the sharp constant.
\end{proof}

Let $z_0$ be a global minimum point of $u^\prime(x)$ as in \cite{John-JDE-1979}. Then in the following proposition, we give a necessary and sufficient condition for the global existence or finite time blowup for the solution of \eqref{CL1} on the characteristic curve $X(t,z_0),$ which satisfies
\begin{equation*}
\frac{\mathrm{d}X(t,z)}{\mathrm{d}t}=u(t,X(t,z)),\quad\quad\quad X(0,z)=z.
\end{equation*}

\begin{proof}[proof of Theorem \ref{thm:2}]
Along the characteristic curve $X(t,z_0),$ let 
\begin{equation*}
U(t)=u(t,X(t,z)).
\end{equation*}
Then the function $U(t)$ satisfies a ordinary differential equation
\begin{equation*}
 \left\{\begin{aligned}
&U^\prime=kU^2, \\
&U(0)=u_0(z_0),
\end{aligned}\right.
\end{equation*}
whose solution is 
\begin{equation}\label{sol-U}
U(t)=\frac{u_0(z)}{1-ku_0(z)t}.
\end{equation}

Along the characteristic curve $X(t,z_0),$ let
\begin{equation*}
W(t)=-u_x(t,X(t,z)).
\end{equation*}
Then 
\begin{equation*}
W^\prime=W^2+2kUW.
\end{equation*}
By \eqref{sol-U}, the function $W(t)$ satisfies a ordinary differential equation
\begin{equation*}
 \left\{\begin{aligned}
&W^\prime=W^2+\frac{2ku_0(z)}{1-ku_0(z)t}W, \\
&W(0)=-u_0^\prime(z_0)=:W_0.
\end{aligned}\right.
\end{equation*}
More precisely, this is a Riccati equation. Let
\begin{equation*}
W(t)=\frac{1}{V(t)}
\end{equation*}
and 
\begin{equation*}
a=-ku_0(z),
\end{equation*}
then 
\begin{equation*}
V^\prime-\frac{2a}{1+at}V=-1,
\end{equation*}
whose solution is 
\begin{equation*}
V(t)=V(0)(1+at)^2-(1+at)t.
\end{equation*}
Hence, 
\begin{equation*}
W(t)=\frac{W_0}{(1+at)(1+(a-W_0)t)}.
\end{equation*}
 
We now proceed to a case-by-case analysis. 

Case 1: Suppose $s_0\leq k^{-1}.$ Then by \eqref{sol-U}, $U(t)$ blows up in finite time for the case $u_0(z)>0.$ When $u_0(z)\leq0,$ the solution $U(t)$ exists globally. Meanwhile, we have
\begin{equation*}
a=k|u_0(z_0)|< ks_0|u^\prime(z_0)|\leq|u^\prime_0(z_0)|=W_0,
\end{equation*}
which implies $W(t),$ or $u_x(t,x)$ blows up on characteristic curve $X(t,z_0)$, and the life span of $u_x$ is
\begin{equation*}
T_{\max}=\frac{1}{W_0-a}.
\end{equation*}
 
Case 2: Suppose $s_0> k^{-1}.$ By Lemma \ref{lem:f}, for any $0<\delta<\frac{s_0}{4}$, there exists a non-identically-zero function $u_{0}(x)$ (by convolution with the function $\tilde{f}_\delta,$ see Picture \ref{pic-f}) whose support is compact and contained in $[\alpha_0,\beta_0]$. This function attains a global minimum at a point $z_0$ and satisfies
\begin{equation*}
u_0(x)\leq0,\quad\quad\quad \forall x\in \mathbb{R}
\end{equation*}
and
\begin{equation}\label{u0}
|u_0(z_0)|\geq (s_0-4\delta)|u^\prime_0(z_0)|.
\end{equation}
Taking $\delta$ small enough, such that $\delta\leq \frac{ks_0-1}{4k},$ then by \eqref{u0} and 
\begin{equation*}
s_0-4\delta\geq k^{-1},
\end{equation*}
we have 
\begin{align*}
a=k|u_0(z_0)|\geq k(s_0-4\delta)|u_0^\prime(z_0)|\geq |u_0^\prime(z_0)|=W_0
\end{align*}
Thus, the functions $U(t)$ and $W(t)$ exist globally, i.e., $u(t,x)$ and $u_x(t,x)$ exist globally on the characteristic curve $X(t,z_0).$
\end{proof}

\section*{Acknowledgments}
Qingsong Zhao was supported by the National Natural Science Foundation of China under Grant Number 12401281.

\end{document}